\documentclass[oneside,notitlepage,12pt]{article}

\usepackage{amssymb}
\usepackage[leqno]{amsmath}
\usepackage{amsfonts}
\usepackage{amsopn}
\usepackage{amstext}
\usepackage{amsthm}
\usepackage{mathtools}

\usepackage{tikz}
\usetikzlibrary{cd}

\usepackage{enumitem}

\usepackage{verbatim}
\usepackage[colorlinks, backref]{hyperref}
\usepackage{makeidx}

\usepackage{calrsfs}
\usepackage{fourier-orns}
\usepackage{clock} 
\usepackage[alpine, weather]{ifsym}

\providecommand{\cal}{\mathcal}
\renewcommand{\Bbb}{\mathbb}

\newenvironment{pf}{\begin{proof}}{\end{proof}}

\newcommand{\Ef}{{\cal{F}}}

\newcommand{\El}{{\cal{L}}}

\newcommand{\Yu}{{\cal{U}}}

\newcommand{\Qyu}{{\Bbb{Q}}}

\newcommand{\sig}{\sigma}

\renewcommand{\phi}{\varphi}
\renewcommand{\rho}{\varrho}

\newcommand{\rest}{\restriction}

\newcommand{\ntr}{{n\in\omega}}

\newcommand{\loe}{\leq}

\newcommand{\subs}{\subseteq}

\newcommand{\nnempty}{\ne\emptyset}

\newcommand{\id}[1]{{\operatorname{i\!d}_{#1}}} 

\newtheorem{tw}{Theorem}[section]
\newtheorem{wn}[tw]{Corollary}
\newtheorem{lm}[tw]{Lemma}
\newtheorem{prop}[tw]{Proposition}

\theoremstyle{definition}
\newtheorem{df}[tw]{Definition}
\newtheorem{ex}[tw]{Example}
\newtheorem{pyt}[tw]{Question}

\theoremstyle{remark}
\newtheorem{uwgi}[tw]{Remark}

\newcommand{\setof}[2]{\{#1\colon #2\}}
\newcommand{\bigsetof}[2]{\Bigl\{#1\colon #2\Bigr\}}

\newcommand{\sett}[2]{\{#1\}_{#2}}
\newcommand{\sn}[1]{\{#1\}} 
\newcommand{\dn}[2]{\{#1,#2\}} 
\newcommand{\pair}[2]{\langle #1, #2 \rangle} 
\newcommand{\map}[3]{#1\colon #2 \to #3} 
\newcommand{\img}[2]{#1[#2]} 

\newcommand{\fra}{Fra\"iss\'e}

\newcommand{\ciag}[1]{{\sett{{#1}_n}{\ntr}}}

\newcommand{\iso}{\approx}

\newcommand{\ciagi}[1]{\sig{#1}}

\newcommand{\bV}{{\mathbb{V}}}
\newcommand{\bW}{{\mathbb{W}}}
\newcommand{\bX}{{\mathbb{X}}}

\newcommand{\bP}{{\mathbb{P}}}

\newcommand{\cmp}{\circ} 

\newcommand{\separator}{\begin{center} \leafright \leafright \leafright \decotwo \decotwo \decotwo \leafleft \leafleft\leafleft
\end{center}}

\providecommand{\ar}{\arrow}

\newcommand{\planes}{\mathcal{P}}
\newcommand{\planesdg}[1]{\planes_{\deg \loe #1}}

\newcommand{\field}[1]{\mathbb{F}_{#1}}

\newcommand{\proplane}[1]{#1\mathbf P^2}

\newcommand{\pl}{\Lambda} 
\newcommand{\pline}{\pl}
\newcommand{\coll}{\pl}

\newtheorem*{clm*}{Claim}
\newenvironment{clmproof}[1][\proofname]{\proof[#1]}{\endproof}

\newcommand{\pfp}{\ensuremath{\dagger}}

\renewcommand{\thefootnote}{\fnsymbol{footnote}}

\title{Homogeneity of abstract linear spaces}
\author{
{\sc Wies{\l}aw Kubi\'s}\footnotemark[1]
\and
{\sc Piotr Nowakowski}\footnotemark[1] \footnotemark[2]
\and
{\sc Tomasz Rzepecki}\footnotemark[3]
}

\date{\clocktime\today}

\begin{document}

\maketitle

\footnotetext[1]{Institute of Mathematics, Czech Academy of Sciences, Czech Republic}
\footnotetext[2]{Faculty of Mathematics and Computer Science, University of Lodz, Poland}
\footnotetext[3]{ORCID: \href{https://orcid.org/0000-0001-9786-1648}{0000-0001-9786-1648}. Mathematical Institute of the University of Wrocław, Poland. The third author is supported by the Narodowe Centrum Nauki grant no. 2016/22/E/ST1/00450.}

\begin{abstract}
	We discuss homogeneity and universality issues in the theory of abstract linear spaces, namely, structures with points and lines satisfying natural axioms, as in Euclidean or projective geometry. We show that the two smallest projective planes (including the Fano plane) are homogeneous and, assuming the continuum hypothesis, there exists a universal projective plane of cardinality $\aleph_1$ that is homogeneous with respect to its countable and finite projective subplanes.
	We also show that the existence of a generic countable
    linear space
    is equivalent to an old conjecture asserting that every finite linear space embeds into a finite projective plane.

	\ \\


\end{abstract}

\tableofcontents

\section{Introduction}
\renewcommand{\thefootnote}{\arabic{footnote}}

A \emph{linear space} is a structure consisting of points and lines, satisfying the obvious axioms: each line passes through at least two points and two different points determine a unique line. The term ``linear space'' may be a bit confusing, as it usually corresponds to vector spaces, however it is already well established in the literature devoted to incidence geometry
\footnote{According to Buekenhout \cite{Buekenhout}, the term \emph{linear space} was coined by Paul Libois in 1961.}, see e.g.~\cite{BaBe, Beutel, Buekenhout, Shult}. We are interested in finite and countable linear spaces, mainly in the context of homogeneity and universality. A linear space $X$ is \emph{homogeneous} if every isomorphism between finite subspaces of $X$ extends to an automorphism of $X$. A linear space $X$ is \emph{universal} for a given class $\Ef$ of spaces if $X \in \Ef$ and every space in $\Ef$ embeds into $X$.

Perhaps the most important and (definitely interesting) examples of linear spaces are \emph{projective planes}, namely, those in which every two lines intersect and
which have at least three points on every line.
Every (countable) linear space embeds into a
(countable) projective plane,
while the question whether every finite linear space embeds into a finite projective plane is an old open problem, see e.g.~\cite{Beutel}. We show, in particular, that a positive answer is implied by the existence of a universal countable
linear space (which, if it exists, is also a universal countable projective plane).
A more precise statement is 
in Theorem~\ref{THMJuniwa} below.

We show that, besides the trivial ones, the only homogeneous linear spaces are the two smallest projective planes $\proplane{\field{2}}$, $\proplane{\field{3}}$, whereas the first one is usually called the \emph{Fano plane}. The full classification is stated in Theorem~\ref{THMKompletJen} below.

Finally, we show that if the continuum hypothesis holds then there exists a universal projective plane of the smallest uncountable cardinality. It is homogeneous with respect to isomorphisms between its countable projective subplanes; this property makes it unique up to isomorphism. The full statement is Theorem~\ref{THMFajwPointFri} below.

\section{Preliminaries}

Linear spaces are typically defined as abstract incidence structures consisting of points and lines, satisfying two natural axioms (as mentioned in the introduction). A formal definition is as follows.

\begin{df}
	A \emph{linear space} is a structure of the form $\pair X \El$, where $X$ is a set, $\El$ is a family of at least 2-element subsets of $X$ satisfying the following axioms.
	\begin{enumerate}
		\item[(L0)] For each $x,y \in X$ there is $L \in \El$ such that $x,y \in L$.
		\item[(L1)] For each $L_0, L_1 \in \El$, there is at most one $x \in L_0 \cap L_1$.
	\end{enumerate}
As usual, the elements of $X$ are called \emph{points} and the elements of $\El$ are called \emph{lines}.
We adopt the usual geometric jargon like ``a line $L$ passes through $x$'' or ``a point $x$ is incident to a line $L$'', both meaning ``$x \in L$''.
\end{df}

It is clear that linear spaces can be equivalently defined as first order structures of the form $\pair X \pl$, where $\pl$ is a symmetric ternary relation representing the collinearity, namely, $\pl(x,y,z)$ means there is a line $L \in \El$ with $x,y,z \in L$. The reader can easily guess the set of axioms for $\pl$ making $\pair X \pl$ a linear space $\pair X \El$, where the lines are of the form
$$\pl(a,b) = \setof{x \in X}{\pl(a,b,x)}$$
with $a \ne b$. From now on, the symbol $\pl$ will have two meanings: a ternary relation representing collinearity and a binary operation defining lines.

\begin{df}
A line $L$ is \emph{trivial} if it consists of precisely two points. A set $S$ is \emph{independent} (with respect to $\pl$) if the collinearity is trivial on $S$, namely, $\pl(x,y,z)$ implies $x=y$ or $x=z$ or $y=z$ for every $x,y,z \in S$ (equivalently, all lines restricted to $S$ are trivial).

A linear space $X$ is \emph{trivial} if $X$ is independent in itself, that is, all lines in $X$ are trivial.
\end{df}

It is clear what a homomorphism of linear spaces should be, namely,
a function
preserving the collinearity relation. An \emph{embedding} of linear spaces is a one-to-one mapping that is at the same time an isomorphism onto its image.

\begin{df}
	A linear space $X$ will be called a \emph{closed plane} if every two lines of $X$ intersect. A linear space $X$ is \emph{non-degenerate} if it contains an independent set consisting of at least $4$ elements. Otherwise it is called \emph{degenerate}.
	A non-degenerate closed plane is called a \emph{projective plane}.
\end{df}

Degenerate linear spaces are those where the collinearity relation is total (all points are collinear) and spaces consisting of a line plus a single point. A trivial linear space is degenerate if it has at most $3$ points.

\begin{df}
	Let $\bX = \pair X \El$ be a linear space. The \emph{degree}
    of a point or a line in $\bX$ is the number of lines and points (respectively) incident to it. The \emph{degree} of $\bX$ is the maximum of degrees of points and lines in $\bX$.
	If $\bX$ is a finite projective plane, its \emph{order} is defined to be $n-1$, where $n$ is the degree of $\bX$.

	Given a natural number $n$, we shall denote by $\planesdg{n}$ the class of all linear spaces of degree at most $n$.
\end{df}

\begin{uwgi}
    It is well-known that in a projective plane, all points and lines have the same degree.
\end{uwgi}

Given a field $F$, the \emph{projective plane} over $F$, denoted by $\proplane{F}$, is the one constructed either by extending the affine plane $F^2$ or, more directly, by taking $1$-dimensional vector subspaces of $F^3$ as points and $2$-dimensional vector subspaces of $F^3$ as lines. Both definitions are well known to be equivalent.
Given a prime $p$ and a positive integer $n$, we denote by $\field{p^n}$ the unique field of order $p^n$. For every prime $p$ and for every positive integer $n$, we have an \emph{algebraic} projective plane $\proplane{\field{p^n}}$ of
order $p^n$ and degree $p^n+1$.

For a thorough treatment of projective geometry, including planes, we refer to Coxeter's monograph~\cite{CoxPP}.

\subsection{Amalgamation and homogeneity}

Let $\Ef$ be a class of mathematical structures in a fixed first order language. The language (and the class) is called \emph{relational} if there are no function symbols in the language (including constants, as $0$-ary functions).
For simplicity, in the definitions below, we assume that the language is relational.
Linear spaces are structures in the language consisting of a single ternary relation, interpreted as collinearity.

An \emph{embedding} of structures is a one-to-one homomorphism that is an isomorphism onto its image.
The following concepts (perhaps except the last
two)
are well known in model theory.

\begin{df}
	We say that $\Ef$ has the \emph{joint embedding property} if for every $X, Y \in \Ef$ there exist embeddings $\map f X V$, $\map g Y V$ with $V \in \Ef$.

	We say $\Ef$ is \emph{hereditary} if for every $X \in \Ef$, every structure isomorphic to a substructure of $X$ is in $\Ef$. In particular, a hereditary class is closed under isomorphisms.

	We say $\Ef$ has the \emph{amalgamation property} at $Z \in \Ef$ if for every two embeddings $\map f Z X$, $\map g Z Y$ with $X,Y \in \Ef$, there exist $W \in \Ef$ and embeddings $\map {f'} X W$, $\map {g'} Y W$ such that $f' \cmp f = g' \cmp g$. In this case, one also says that $Z$ is an \emph{amalgamation base} in $\Ef$.

	We say that $\Ef$ has the \emph{amalgamation
	property} if every $Z \in \Ef$ is an amalgamation base.

	We say that $\Ef$ has the \emph{cofinal amalgamation property} if for every $Z \in \Ef$ there is an embedding $\map e Z {Z'}$ such that $Z'$ is an amalgamation base in $\Ef$.

	Finally, we say that $\Ef$ has the \emph{weak amalgamation property} if for every $Z \in \Ef$ there is an embedding $\map e Z{Z'}$ with $Z'\in \Ef$, such that for every two embeddings $\map f {Z'} X$, $\map g {Z'} Y$ with $X,Y \in \Ef$, there exist $W \in \Ef$ and embeddings $\map {f'} X W$, $\map {g'} Y W$ such that $f' \cmp f \cmp e = g' \cmp g \cmp e$.
\end{df}

The weak amalgamation property, which is obviously the weakest one of the above, was introduced by Ivanov~\cite{Ivanov} in the context of generic automorphisms. Classes with the weak amalgamation property have been studied and characterized by Krawczyk and the first author~\cite{KraKub}, exhibiting connections with the universality number.

We now briefly describe the story of generic limits, including \fra\ theory.
Namely, given a class $\Ef$ of finite structures with the joint embedding property and the weak amalgamation property, there exists a unique (up to isomorphism) structure
$$U \in \ciagi{\Ef} := \bigsetof{\bigcup_{\ntr}X_n}{\ciag X \text{ is a chain in }\Ef}$$
that is called \emph{generic} over $\Ef$ and can be described as the structure obtained through a strategy in the natural game (called \emph{abstract Banach-Mazur game} \cite{KraKub}) in which two players build a chain of $\Ef$-structures, at each step extending the structure arbitrarily.
It turns out that the existence of a generic structure implies the weak amalgamation property.
Furthermore, if $\Ef$ fails the weak amalgamation property then the universality number of $\ciagi \Ef$ is the continuum. In other words, the minimal cardinality of a family $\Yu \subs \ciagi \Ef$ such that every $X \in \ciagi{\Ef}$ embeds into some $U \in \Yu$ is the continuum. See~\cite{KraKub} for details.

Once the class $\Ef$ has the amalgamation property and is hereditary, the generic structure is called the \emph{\fra\ limit} of $\Ef$ and is homogeneous with respect to $\Ef$ and universal in $\ciagi \Ef$.
Cofinal amalgamation (which is actually relevant in our considerations) is not good enough for universality: negative examples can be found in~\cite{KraKub}.

\separator

A mathematical structure $U$ is said to be \emph{homogeneous} if every isomorphism between its ``small'' substructures extends to an automorphism of $U$. Here, ``small'' typically means ``finite'' or ``finitely generated''\footnote{Homogeneity with respect to finite or finitely generated substructures is also called \emph{ultra-homogeneity}.}, although in general it could have some other meanings. For instance, in topology, homogeneity is typically understood as point-homogeneity, namely, asking that the auto-homeomorphism group (or its fixed subgroup) acts transitively on the space.

In this paper, we shall consider homogeneity with respect to finite substructures, specifically, finite subspaces of linear spaces.

Homogeneity is closely connected to amalgamation, which is actually the essence of \fra\ theory of universal homogeneous structures, see~\cite{Fraisse}, \cite{Hodges}, and for a category-theoretic treatment~\cite{Kub40}.
Finite homogeneous relational structures are easily described as follows.

\begin{prop}\label{PRPterminals}
	Assume $\Ef$ is a hereditary class of finite structures with the joint embedding property and the amalgamation property. Assume $U \in \Ef$ is such that every embedding of $U$ into any structure in $\Ef$ is an isomorphism.
	Then $U$ is homogeneous, unique up to isomorphism, and universal in $\Ef$.
\end{prop}

\begin{pf}
	Fix embeddings $\map f A U$, $\map g B U$ with $A,B \in \Ef$ and an isomorphism $\map h A B$.
	The amalgamation property gives embeddings $\map {f'} U W$, $\map {g'} U W$ such that $W \in \Ef$ and $f' \cmp f = g' \cmp g \cmp h$.
	By the assumption, $f'$, $g'$ are isomorphisms.
	Thus $(g')^{-1} \cmp f' \cmp f = g \cmp h$. In case $f,g$ are inclusions $A \subs U$, $B \subs U$, this evidently shows that $h$ extends to an automorphism $(g')^{-1} \cmp f'$. This proves that $U$ is homogeneous.

	Uniqueness follows immediately from the joint embedding property: Given $U, U'$ as above, there are embeddings $\map i U W$, $\map j {U'} W$ with $W \in \Ef$, that are actually isomorphisms, showing that $U \iso U'$.

	Finally, given $A \in \Ef$, again by the joint embedding property, there are embeddings $\map e A W$ and $\map f U W$ with $W \in \Ef$. Again, $f$ is an isomorphism, therefore $f^{-1} \cmp e$ is an embedding of $A$ into $U$. This proves universality.
\end{pf}

We shall use the fact above in the next section, proving homogeneity of the two smallest projective planes.

\section{Classifying homogeneous linear spaces}\label{SECThomFanos}

We show that, among countable linear spaces, homogeneity is a rare property. Namely, besides trivial cases, the only homogeneous linear spaces are the projective planes $\proplane{\field{2}}$ and $\proplane{\field{3}}$.

The nontrivial degenerate linear spaces are of the following forms.
\begin{enumerate}[itemsep=0pt]
	\item[(L)] A single line: all points are collinear. This includes all spaces with at most $2$ points.
	\item[(L$^+$)] A line + a point: a line plus a single point outside of this line.
\end{enumerate}

The only homogeneous ones are of course in (L). Moreover, every trivial linear space is homogeneous.
The next statement gives a strong restriction on homogeneity of general linear spaces.

\begin{lm}
	Let $\bX = \pair{X}{\El}$ be a homogeneous nontrivial linear space which is not a line. Then $\bX$ is a projective plane.
\end{lm}

\begin{pf}
	First, observe that by homogeneity, for every two lines $L_0,L_1 \in \El$, there is an automorphism $\map h \bX \bX$ such that $\img h{L_0} = L_1$ (take a bijection between fixed two-element subsets of $L_0$, $L_1$ and extend it to an automorphism).
	This shows that each line has the same cardinality, which is $>2$, because $\bX$ is nontrivial.
	Since there are at least three non-collinear points, there exist lines $L_0, L_1$ with $L_0 \cap L_1 = \sn p$.
	Choose $a_0, a_1 \in L_0 \setminus \sn p$, $b_0, b_1 \in L_1 \setminus \sn p$ with $a_0 \ne a_1$, $b_0 \ne b_1$.
	Note that the set $\{a_0,a_1,b_0,b_1\}$ is independent, so in particular $\bX$ is a non-degenerate.

	Now, given two different lines $L_2, L_3$ with two pairs of points $x_0,x_1 \in L_2$, $y_0,y_1 \in L_3$, assuming $\{x_0,x_1\} \cap \{y_0,y_1\} = \emptyset$, we have an isomorphism
	$$\map{f}{\{a_0,a_1,b_0,b_1\}}{\{x_0,x_1,y_0,y_1\}}$$
	defined by $f(a_i) = x_i$, $f(b_i) = y_i$ for $i=0,1$. An automorphism extending $f$ maps $L_0$ onto $L_2$ and $L_1$ onto $L_3$, witnessing that $L_2 \cap L_3 \nnempty$. Thus $\bX$ is a projective plane.
\end{pf}

Recall that the degree of a plane is the maximal number $n$ such that every line has at most $n$ points and every point belongs to at most $n$ lines. On the other hand, the \emph{order} of a projective plane is the unique number $k$ such that every line has $k+1$ many points.

\begin{lm}
    \label{lm:homogenous_plane_degree_4}
	A homogeneous projective plane is finite of degree $\loe 4$.
\end{lm}

\begin{pf}
	Fix a projective plane $\bP = \pair P \El$ of degree $> 4$. Fix $H \in \El$ and let us call it the \emph{horizon}.
	Choose $s \ne t$ in $H$.
	Let $L_0,L_1,L_2 \in \El \setminus \sn H$ be three different lines passing through $s$.
	Let $K_0, K_1 \in \El \setminus
	\sn H$ be two different lines passing through $t$.
	Let $S \subs P \setminus H$ be a six-point set consisting of all intersections of the lines above, outside of the horizon.
	Specifically, $S$ contains precisely the points $x_{i,j} \in L_i \cap K_j$, $i<3$, $j<2$.
	Then $S$ is a subplane of $\bP$, where $\{x_{0,0},x_{1,0},x_{2,0}\}$, $\{x_{0,1},x_{1,1},x_{2,1}\}$ are the only non-trivial lines and they are parallel.
	Since the degree of $\bP$ is $> 4$, there exists a point $y \in K_0 \setminus \{x_{0,0}, x_{1,0}, x_{2,0}, t\}$.

	The map $\map f S {P}$ defined by $f(x_{0,0}) = y$ and $f \rest S \setminus \sn{x_{0,0}} = \id{S \setminus \sn{x_{0,0}}}$ is an isomorphism onto its image
    $(S \cup \sn y) \setminus \sn{x_{0,0}}$. Suppose $h$ is an automorphism of $\bP$ extending $f$.
	Then $h(s) = s$, because the lines $L_1$, $L_2$ are invariant with respect to $h$. On the other hand, in $P\setminus H$, the line $R$ passing through $y, x_{0,1}$ cannot be parallel to $L_1$, because it intersects $L_0$ (parallelism in $P\setminus H$ is an equivalence relation), therefore $R$ intersects $L_1$ at some point of $P \setminus H$ (in particular, $s \notin R$), while at the same time $\img h {L_0} = R$. This is a contradiction.
\end{pf}

The smallest projective plane of degree $> 4$ is the one based on $\field{4}^2$, where $\field{4}$ is the unique field of cardinality $4$.
In order to get a complete classification of homogeneous linear spaces, it remains to check the two smallest projective planes. They are indeed homogeneous, as we shall see in a moment.

\subsection{The Fano plane}

The smallest projective plane $\proplane{\field{2}}$ consists of seven points and is well known under the name \emph{the Fano plane} as it is a unique projective plane of degree $3$ and order $2$.

Our goal is to prove the following:

\begin{tw}\label{THMfanoAP}
	Class $\planesdg{3}$ has the amalgamation property.
\end{tw}

It is clear that the Fano plane does not embed into any bigger linear space of degree at most $3$, therefore Theorem~\ref{THMfanoAP} together with Proposition~\ref{PRPterminals} lead to:

\begin{wn}
	The Fano plane is homogeneous.
\end{wn}

The above corollary can also be proved by ``brute force'' (there are not so many cases to check) however we find our arguments much more illustrative. Moreover, Theorem~\ref{THMfanoAP} brings a little bit more information, saying that every plane of degree $\loe3$ embeds into the Fano plane. This fact can also be proved directly, by checking all possible isomorphic types.

The proof of Theorem~\ref{THMfanoAP} will be preceded by a simple lemma.

\begin{lm}\label{LMdgoudgo}
	Let $\bX \in \planesdg{3}$ and let $L_0 \in \El^\bX$. Then there exists at most one line parallel to $L_0$ and once this happens, $|L_0|=2$.
\end{lm}

So, if $|L_0|=3$, then no other line is parallel to $L_0$.

\begin{pf}
	First, suppose $|L_0|=3$ and $L_1$ is a line parallel to $L_0$. Choose $x \in L_1$. For each $p \in L_0$ there is a unique line $\pl(x,p)$ passing through $x$, and it is
    distinct
    from $L_1$. Thus the degree of $x$ is at least $4$, a contradiction.

	Now suppose $L_0 = \dn ab$ and suppose $L_1, L_2$ are two different lines parallel to $L_0$.
	Notice that $L_1$ and $L_2$ must be parallel, because if $p \in L_1 \cap L_2$ then the lines $L_1, L_2, \pl(p,a), \pl(p,b)$ would witness that the degree of $p$ is at least $4$.

	Assume $c,d \in L_1$, $c \ne d$.
	The lines $L_0, \pl(a,c), \pl(a,d)$ are pairwise different, so the degree of $a$ is $3$ and there can be no other point on $L_1$ (also by the previous argument).
	The same applies to $b$ and $L_2$.

	Thus $L_1 = \dn cd$, $L_2 = \dn ef$ and, renaming the points if necessary, we may assume $a,c, e$ and $a,d,f$ are collinear (if, say, $a,c,e$ were not collinear then the degree of $a$ would be at least $4$).

	Finally, $\{a,c,e\}$, $\pl(b,f)$ are parallel, which contradicts the first part.
\end{pf}

\begin{pf}[Proof of Theorem~\ref{THMfanoAP}]
	It is enough to show that, given $\bX = \pair X \El \in \planesdg{3}$, every two one-point extensions of it can be amalgamated in $\planesdg{3}$. Fix one-point extensions $\bX_a, \bX_b \in \planesdg{3}$. We have two sets $X \cup \sn a$, $X \cup \sn b$, each with a
    linear space
    structure extending the one of $\bX$.
	We may assume that $\bX_a$, $\bX_b$ are not isomorphic over $\bX$, since otherwise the amalgamation is trivial, by identifying $a$ and $b$.
	In other words, we assume that the line structures of $\bX_a, \bX_b$ are different.

	We say $a$ is \emph{free} over $\bX$ if it is not collinear with any pair of distinct points of $X$. Otherwise, we say $a$ is determined by a line $L \in \El$, if $a$ is collinear with some $x_0 \ne x_1$, $x_0, x_1 \in L$. Note that in this case $L = \dn{x_0}{x_1}$, because the degree (cardinality) of each line in $\bX_a$ is $\loe3$. Of course, the same definition of being free or determined applies to $b$.

	If both $a$, $b$ are free over $\bX$ then $\bX_a$, $\bX_b$ are isomorphic over $\bX$ and there is nothing to prove.
    Thus, we may assume that $b$ is determined by the line $\dn{c_0}{c_1} \in \El$. We are now left with two cases.

	\textbf{Case $1$:} $a$ is free over $\bX$.

	If $X = \dn{c_0}{c_1}$, then the free amalgamation of $\bX_a, \bX_b$ (that is, not adding any unnecessary lines) is in $\planesdg{3}$. Otherwise, we claim that $X = \{c_0, c_1, d\}$ (and $c_0,c_1,d$ are of course not collinear). Indeed, supposing there is $x \in X \setminus \{c_0,c_1,d\}$, we would have at least four lines
	\[\pl(c_0,a), \pl(c_1,a), \pl(d,a), \pl(x,a)\]
	passing through $a$, a contradiction (recall that $a$ is free over $\bX$, so $x,d,a$ are not collinear).
	Finally, knowing that $X = \{c_0,c_1,d\}$, it suffices to make $a,b,d$ collinear in order to get a plane of degree $\loe3$.

	\textbf{Case $2$:} $a$ is not free over $\bX$.

	Assume $a$ is determined by the line $\dn{d_0}{d_1}$.
	Suppose first that the lines $\{c_0,c_1\}$, $\{d_0,d_1\}$ are parallel.
	Then $a$ is also determined by $\dn{c_0}{c_1}$, since otherwise the lines $\{d_0,d_1,a\}, \{c_0,c_1\}$ would be parallel in $\bX_a$, contradicting Lemma~\ref{LMdgoudgo}. By the same reason, $b$ is determined by both lines $\dn{c_0}{c_1}$ and $\dn{d_0}{d_1}$.
	Again by Lemma~\ref{LMdgoudgo}, neither $a$ nor $b$ is determined by a line different from $\{c_0,c_1\}$, $\{d_0,d_1\}$. Finally, $\bX_a, \bX_b$ are isomorphic over $\bX$ and we amalgamate them by identifying $a$ and $b$.

	Suppose now that $\{c_0,c_1\} \cap \{d_0,d_1\} \nnempty$.
    We may assume that $c_0 = d_0$.
	If $X = \{c_0,c_1,d_1\}$, the free amalgamation of $\bX_a$, $\bX_b$ is in $\planesdg{3}$. So suppose
	\[X = \{c_0,c_1,d_1,x,y\},\]
    where possibly $x=y$.
	Note that $c_0,x,y$ are collinear, since otherwise the degree of $c_0$ would be $>3$.
	If $x\ne y$ and neither $x$ nor $y$ is collinear with $c_1,d_1$, then $\{c_1,d_1\}$ and $\{c_0,x,y\}$ would be parallel lines in $\bX$, contradicting Lemma~\ref{LMdgoudgo}.
	Assuming $x$ is collinear with $c_1,d_1$ (and either $y=x$ or $y \ne x$), we declare $a,b,x$ to be collinear.
	Finally, in both cases the amalgamation is in $\planesdg{3}$.
	This completes the proof.
\end{pf}

\subsection{The projective plane of degree \texorpdfstring{$4$}{4}}

We are going to show that $\proplane{\field{3}}$, the unique projective plane of degree $4$ and order $3$, is homogeneous. First, we identify the class of all substructures of $\proplane{\field{3}}$.

A linear space consisting of exactly five points with only trivial lines is called the \emph{pentagon}.
By $\planesdg{4}^*$ we denote $\planesdg{4}$ with both the Fano plane and the pentagon excluded.
Note that the projective plane $\proplane{\field{3}}$ is in $\planesdg{4}^*$ and it cannot be embedded into any bigger linear space in $\planesdg{4}^*$, therefore by Proposition~\ref{PRPterminals}, once we prove that $\planesdg{4}^*$ has the amalgamation property, we obtain the homogeneity of $\proplane{\field{3}}$.

\begin{tw}\label{THM4AP}
	Class $\planesdg{4}^*$ has the amalgamation property.
\end{tw}
\begin{pf}
	It is enough to show that, given $\bX = \pair X \El \in \planesdg{4}^*$, every two one-point extensions of it can be amalgamated in $\planesdg{4}^*$. Fix one-point extensions $\bX_a, \bX_b \in \planesdg{4}^*$. We have two sets $X \cup \sn a$, $X \cup \sn b$, each with a plane structure extending the one of $\bX$.
	We may assume that $\bX_a$, $\bX_b$ are not isomorphic over $\bX$, since otherwise the amalgamation is trivial, by identifying $a$ and $b$.
	In other words, we assume that the line structures of $\bX_a, \bX_b$ are different.


	If both $a$, $b$ are free over $\bX$, then $\bX_a$, $\bX_b$ are isomorphic over $\bX$ and there is nothing to prove. Thus, we may assume $b$ is determined by the line
    $\pl(c_0, c_1) \in \El$. We are now left with the following cases.

	\textbf{Case $1$:} $a$ is free over $\bX$.

	Then $X$ may have at most four points.
	If $X$ consists of at most $3$ points, then the free amalgamation of $\bX_a, \bX_b$ (that is, not adding any unnecessary lines) is in $\planesdg{4}^*$.
	Suppose that $X$ has four points $c_0,c_1,x,y$. Observe that if $\pl(c_0,c_1)$ consists of exactly two points in $X$, then then the line $\pl(x,y)$ must consist of three points. Indeed, if $\pl(x,y)$ consists of only two points, then, since $a$ is free over $\bX$, the plane $\bX_a$ would be the pentagon, which is not in $\planesdg{4}^*$,
    so we may assume that $c_0\in \pl(x,y)$.
    Now it suffices to  declare $x,a,b$ (or $y,a,b$) to be collinear to obtain the amalgamation in $\planesdg{4}^*$.

	Now, assume that there is a third point in $X$, which lies on $\pl(c_0,c_1)$. Since $a$ is free over $X$, then there may be at most one more point in $X$ (denote it by $x$). Declaring $x,a,b$ to be collinear we obtain the amalgamation in $\planesdg{4}^*$.

	\textbf{Case $2$:} $a$ is determined by the line
    $\pl(d_0, d_1)$, 
    $c_i \neq d_j$ for $i,j<2$ and neither $a$ nor $b$ is determined by both lines.

	Using the fact that $\bX_a$ and $\bX_b$ are in $\planesdg{4}^*$, we infer that every point from the set $\{c_0,c_1,d_0,d_1\}$ lies on four lines either in $\bX_a$ or $\bX_b$, so it cannot lie on any other line. So, the only other points, which can be in $X$ are the following:
	\begin{itemize}[itemsep=0pt]
		\item the common point $x_0$ of lines $\pl(c_0,d_0)$ and $\pl(c_1,d_1)$;
		\item the common point $x_1$ of lines $\pl(c_0,d_1)$ and $\pl(c_1,d_0)$;
		\item the point $y_0$, which lies on the line $\pl(c_0,d_0)$, but does not lie on the line $\pl(c_1,d_1)$;
		\item the point $y_1$, which lies on the line $\pl(c_1,d_1)$, but does not lie on the line $\pl(c_0,d_0)$;
		\item the point $z_0$, which lies on the line $\pl(c_0,d_1)$, but does not lie on the line $\pl(c_1,d_0)$;
		\item the point $z_1$, which lies on the line $\pl(c_1,d_0)$, but does not lie on the line $\pl(c_0,d_1)$;
		\item the common point $w$ of lines $\pl(c_0,c_1)$ and $\pl(d_0,d_1)$;
	\end{itemize}
	Observe that if $x_0, x_1 \in X$, then both $a$ and $b$ must be determined by the line $\{x_0,x_1\}$. Otherwise, the degree of $a$ or $b$ in $\bX_a$ or $\bX_b$ would be $>4$. Hence we need to declare $a,b,x_i$ to be colinear, whenever $x_i \in X$ for $i<2$.
	(Also, since $\bX,\bX_a,\bX_b \in \planesdg{4}^*$, we have collinearity of the following points (if they appear) in appropriate planes and as a result in the amalgamation: $\{c_0,z_1,a,y_1\}$, $\{c_1,z_0,a,y_0\}$, $\{d_1,b,z_1,y_0\}$, $\{d_0,b,z_0,y_1\}$, $\{y_0,y_1,x_1,w\}$, $\{x_0,z_0,z_1,w\}.$)
	Finally, whatever points other than $c_0,c_1,d_0,d_1$ are in $X$, we have that the amalgamation is in $\planesdg{4}^*$.

	\textbf{Case $3$:} $a$ is determined by the line $\pl(c_0, c_1)$ and neither $a$ nor $b$ is determined by any other line.

	Then $\bX_a$, $\bX_b$ are isomorphic over $\bX$ and there is nothing to prove.

	\textbf{Case $4$:} $a$ is determined by the lines $\pl(c_0, c_1)$ and $\pl(d_0, d_1)$ and $b$ is also determined by $\pl(d_0, d_1)$, where $c_i \neq d_j$ for $i,j<2.$

	Then again $\bX_a$, $\bX_b$ are isomorphic over $\bX$.

	\textbf{Case $5$:} $a$ is determined by the lines $\pl(c_0, c_1)$ and $\pl(d_0, d_1)$ and $b$ is not determined by $\pl({d_0, d_1})$, where $c_i \neq d_j$ for $i,j<2$.

	If $b$ is determined by some other line, then we have the situation as in case $2$. So, assume that $b$ is not determined by any other line. Then there is at most one more point $\{x\}$ in $X$ other than $c_0,c_1,d_0,d_1$. Since degree of $d_0$ and $d_1$ in $\bX_b$ is $4$, then $x$ may lie on  the line $\pl(d_0,d_1)$ or be the intersection point of lines $\pl(c_0,d_0),\pl(c_1,d_1),$ or be the intersection point of lines $\pl(c_0,d_1),\pl(c_1,d_0).$ In any case the amalgamation with collinearity of points $a,b,c_0,c_1$ is in $\planesdg{4}^*$.
	
	\textbf{Case $6$:} $a$ is determined by the line $\pl({c_0, d_1})$ and $d_1 \neq c_1$.
	%

	First, suppose that there is the fourth point $c_2$ (other than $b$) on the line $\pl(c_0,c_1)$. But then $b$ is on the line $\pl(c_1,c_2)$, so we have a situation like in Case 2.
	Similar reasoning can be repeated when there is fourth point $d_2$ (other than $a$) on the line $\pl(c_0,d_1)$.

	Now, suppose that there is no fourth point on any of the lines $\pl(c_0,c_1), \pl(c_0,d_1)$. Then the only other points, which can be in $X$ are the following:
	\begin{itemize}[itemsep=0pt]
		\item the point $x$ which lies on the lines $\pl(c_1,a)$ (in $\bX_a$) and $\pl(d_1,b)$ (in $\bX_b)$;
		\item the point $y_0$ which lies on the line $\pl(c_1,a)$ (in $\bX_a$), but does not lie on the line $\pl(d_1,b)$ (in $\bX_b)$;
		\item the point $y_1$ which lies on the line $\pl(d_1,b)$ (in $\bX_b$), but does not lie on the line $\pl(c_1,a)$ (in $\bX_a)$;
		\item points $z_0$ and $z_1$ which lie on the line $\pl(c_1,d_1)$;
		\item the point $w$, which does not lie on the line $\pl(c_1,d_1)$ and also does not lie neither on the line $\pl(d_1,b)$ (in $\bX_b)$ nor on the line $\pl(c_1,a)$ (in $\bX_a)$.
	\end{itemize}
	Suppose that $x \in X$. If $y_0 \in X$ or $y_1 \in X$, then we have a situation like in case $2$. So, suppose that $y_0,y_1 \notin X$. If we have a point on $\pl(c_1,d_1)$ which lies on $\pl(c_0,x)$, then call it $z_0$ (the other point on $\pl(c_1,d_1)$ is called $z_1$). If also $w \in X$, then it is easy to see that $w \in \pl(c_0,x)$ (otherwise $w$ would have degree $>4$ in $\bX_a$). Observe that if $z_1,w \in X$, then $a$ and $b$ must be determined by $\pl(z_1,w)$ (otherwise $w$ would have degree $>4$ in $\bX_a$). Thus, we declare points $z_1,w,a,b$ to be collinear in the amalgamation (even if only one of $z_1$, $w$ is in $X$). Then the amalgamation is in $\planesdg{4}^*$, in particular, it is not the Fano plane (we do not declare $a,b,z_0$ to be collinear).

	Now, suppose that $x \notin X$. If also $y_0,y_1 \notin X$, then we amalgamate as above. If both $y_0,y_1 \in X$, then we have a situation like in case $2$. So, without loss of generality, assume that $y_0 \in X$ and $y_1 \notin X$. If we have a point on $\pl(c_1,d_1)$ which lies on $\pl(c_0,y_0)$, then call it $z_1$ (the other point on $\pl(c_1,d_1)$ is called $z_0$). As before, we have that points $c_0,z_0,w$ are collinear (if they all appear in $X$) and that if $z_1,w \in X$, then $a$ and $b$ must be determined by $\pl(z_1,w)$ (otherwise, $a$ or $b$ would lie on more than $4$ lines). Thus, we again declare points $z_1,w,a,b$ to be collinear in the amalgamation (even if only one of $z_1$, $w$ is in $X$). It is easy to see that if $z_0 \in X$, then $b$ must be determined by $\pl(z_0,y_0)$  (otherwise, $y_0$ would lie on more than $4$ lines). Also, if $w \in X$, then
$w,d_1,y_0$ are collinear (otherwise, again $y_0$ would lie on more than $4$ lines). Finally, we have that the amalgamation is in $\planesdg{4}^*$, which finishes the proof.
\end{pf}

\begin{wn}
	The projective plane $\proplane{\field{3}}$ is homogeneous.
\end{wn}

\subsection{The classification}

Now we know that the only homogeneous nontrivial and non-degenerate linear spaces are the two smallest projective planes $\proplane{\field{2}}$ and $\proplane{\field{3}}$. This leads to the final classification result.

\begin{tw}\label{THMKompletJen}
	Homogeneous linear spaces are one of the following form.
	\begin{enumerate}[itemsep=0pt]
		\item[{\rm(T)}] Trivial spaces (the collinearity is trivial).
		\item [{\rm(L)}] Lines (the collinearity is total).
		\item[{\rm(P$_2$)}] The projective plane $\proplane{\field{2}}$.
		\item[{\rm(P$_3$)}] The projective plane $\proplane{\field{3}}$.
	\end{enumerate}
\end{tw}

Every linear space has its dual, obtained by interchanging points and lines. For example, the dual of a space consisting of a single line is a pencil of trivial lines, namely a set of lines passing through a common point. Of course, this is not homogeneous anymore, as we define homogeneity with respect to the set of points.
We may also look at a stronger homogeneity, when linear spaces are viewed as bipartite graphs. Here the homogeneity easily fails, as the following example shows.

\begin{ex}
	Let $X$ be a linear space with $3$ collinear points $a,b,c$ and with another point $d$ not collinear with $a,b,c$. We have a bijection
    $\map{f}{\{a,b,c\}}{\{a, b, d\}}$.
	Assuming $X$ is viewed as a bipartite graph and $a,b,c,d$ are vertices on the left-hand side, $f$ is a partial isomorphism. Obviously, it cannot be extended to an automorphism of $X$, as the collinearity relation becomes relevant.
\end{ex}

The next question is which finite linear spaces can be amalgamation bases. We give the complete answer in the next section. There is a rather old conjecture that, to our knowledge, is still open.

\begin{enumerate}
	\item[(\pfp)] Every finite linear space embeds into a finite projective plane.
\end{enumerate}

It is not clear who first posed this as an open problem. For convenience, we will call it \emph{Conjecture} (\pfp).

In the next section we show that this conjecture is equivalent to the weak amalgamation property of finite linear spaces and therefore its possible negative answer would imply that there is no universal countable projective plane.


\section{Amalgamating finite linear spaces}

The goal here is to show that the amalgamation bases in the class of finite linear spaces are exactly the finite projective planes (Proposition~\ref{proposition:projective_iff_amalgamation_base}), from which we deduce that the cofinal, as well as weak, amalgamation for the class of finite linear spaces is equivalent to the statement that every such space can be embedded in a finite projective plane (Theorem~\ref{theorem:wap_cap_planes}) which, to our knowledge, remains an open problem in incidence geometry \cite{Beutel}.

\begin{df}
	If $f\colon A\to B$ is an embedding of linear spaces, then we say that:
	\begin{itemize}
		\item
		$f$ is an \emph{elementary planarisation} if $B\setminus A= \{b\}$ and some two lines in $A$ intersect at $b$,
		\item
		$f$ is a \emph{trivial elementary planarisation} if it is elementary and the new point $b$ lies on exactly two lines existing in $A$ (i.e.\ there are two lines $L_1, L_2$ in $A$ such that $L_1\cup \{b\}$, $L_2\cup \{b\}$ are lines in $B$ and all other lines in $B$ containing $b$ are trivial),
		\item
		$f$ is a \emph{planarisation} if there is a (possibly transfinite) sequence $A\to B_0\to B_1\to \ldots$ of elementary planarisations whose colimit is $B$,
		\item
		$f$ is a \emph{trivial planarisation} if its constituent elementary planarisations are all trivial,
		\item
		$f$ is a \emph{complete planarisation} if $f$ is a planarisation and $B$ is a projective plane (equivalently, if it is a maximal planarisation),
		\item
		a \emph{planar closure} of $A$ in $B$ is a maximal planarisation $A\to \bar A$ through which $f$ factors,
		\item
		$f$ is an \emph{aplanar} embedding if $A$ is the planar closure of $A$ in $B$ (equivalently, any two lines in $A$ intersecting in $B$ already intersect in $A$).
	\end{itemize}
\end{df}

\begin{uwgi}
	The planar closure of $A$ in $B\supseteq A$ can be recursively constructed as follows: start with $A_0=A$, then for each $n$, let $A_{n+1}$ be the union of $A_n$ and all the intersections of lines in $A_n$ in $B$, and put $\bar A\coloneqq\bigcup_n A_n$. In particular, planar closures always exist and are unique.
\end{uwgi}

\begin{uwgi}
	A composition of planarisations is a planarisation. In particular, if $A\to \bar A$ is the planar closure of $A\to B$, then the factor map $\bar A\to B$ is aplanar, and planar closure is a closure operator.
\end{uwgi}

\begin{prop}
	\label{proposition:trivial_step}
	If $(L_1 L_1'), (L_2, L_2'), \ldots, (L_n, L_n')$ are distinct parallel pairs in a linear space $A$, then there is a trivial planarisation $A'=A\cup (a_1,\ldots, a_n)$ such that for each $i$, $a_i$ lies on the extensions of $L_i$ and $L_i'$ and all other lines containing any $a_i$ (including $\pline(a_i, a_j)$ for any $j\neq i$) are trivial. In particular, if $(L, L')$ is a pair of parallel lines in $A$ distinct from $(L_i, L_i')$ for all $i$, then $L, L'$ do not intersect in $A'$.
\end{prop}
\begin{proof}
	The definiton of the structure is obvious. Correctness is straightforward: when checking transitivity, we need to check collinear triples with two points in common. But the lines connecting the new points are all trivial, so all such triples will have two points already present in $A$, so transitivity easily follows from transitivity in $A$ (and the assumption that the pairs of lines intersecting at the new points were parallel in $A$).
\end{proof}

\begin{uwgi}
	Using Proposition~\ref{proposition:trivial_step} one can show that every linear structure $A$ admits a complete trivial planarisation $\tilde A$ (which is infinite, if $A$ is not already a projective plane); see Lemma~\ref{LMghyio}.
\end{uwgi}

\begin{prop}
	\label{proposition:intersect_three_lines}
	If $L_1, L_2, L_3, \ldots, L_n$ are pairwise parallel lines in $A$ (with $n\geq 3$), then we can find a (nontrivial) planarisation $A'=A\cup \{a\}\supseteq A$ such that the lines containing $L_i$ are concurrent in $A'$ (namely, they contain $a$).
\end{prop}

\begin{proof}
	Similarly to Proposition~\ref{proposition:trivial_step}, the proof is straightforward.
\end{proof}

\begin{lm}
	\label{lemma:planar_epimorphism}
	Planarisations are epimorphisms in the category of linear spaces with embeddings, i.e.\ if $f\colon A\to A'$ is a planarisation and $g_1,g_2\colon A'\to C$ are embeddings such that $g_1\circ f= g_2\circ f$, then $g_1=g_2$.
\end{lm}
\begin{proof}
	Suppose for simplicity that $f$ is an inclusion. If $A=A'$, the conclusion is trivial.

	Otherwise, if $A'=A\cup \{a\}$ (so that $f$ is an elementary planarisation) and $L_1,L_2$ are lines in $A$ intersecting at $a$ in $A'$. Note that $g_1(a)$ is the intersection of the unique lines in $C$ containing $g_1[L_1]=g_1\circ f[L_1]=g_2\circ f[L_1]=g_2[L_1]$ and $g_1[L_2]=g_2[L_2]$, as is $g_2(x)$, so $g_1(a)=g_2(a)$, whence $g_1=g_2$.

	Finally, if $f$ is a non-elementary planarisation, then the conclusion follows by an easy inductive argument.
\end{proof}

\begin{uwgi}
	\label{remark:lifting_amalgams}
	Note that Lemma~\ref{lemma:planar_epimorphism} easily implies that if $A\to A'$ is a planarisation and $B_1, B_2$ are extensions of $A'$, then any amalgam $C$ of $B_1, B_2$ over $A$ is also an amalgam over $A'$: in the diagram below, $g_1\circ f_1\circ f=g_2\circ f_2\circ f$, and since $f$ is an epimorphism, $g_1\circ f_1=g_2\circ f_2$.
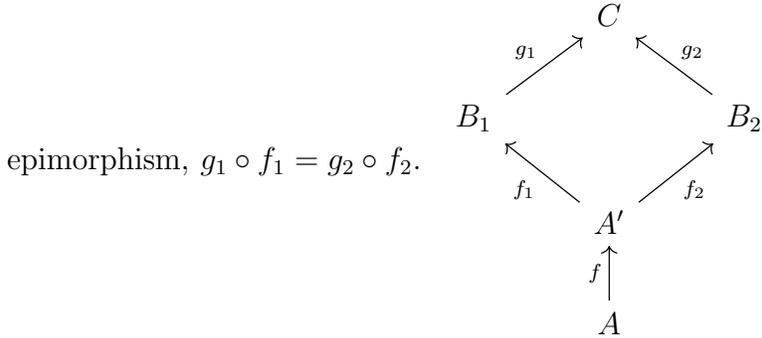
	\begin{raggedright}
		\begin{tikzcd}
			&C&\\
			B_1\ar[ur,"g_1"]&&B_2\ar[ul,"g_2", swap]\\
			&A'\ar[ul,"f_1"]\ar[ur,"f_2", swap]\\
			&A\ar[u,"f"]
		\end{tikzcd}
	\end{raggedright}
\end{uwgi}

\begin{lm}
	\label{lemma:nontrivial_incomaptible}
	If $A$ is a finite linear space and $A\to A'$ is a finite nontrivial planarisation, then we can find a finite trivial planarisation $A\to A''$ such that $A', A''$ cannot be amalgamated over $A$.
\end{lm}
\begin{proof}
	We may assume without loss of generality that $A\subseteq A'$. Suppose first that $A'$ is elementary, i.e.\ $A'\setminus A=\{a'\}$. Since $A'$ is nontrivial, $a'$ is the intersection of at least three lines in $A$, say $L_1,L_2,L_3$. Let $A''\cup \{a''\}$ be a trivial elementary planarisation in which $L_1$ and $L_2$ intersect (which we have by Proposition~\ref{proposition:trivial_step}). Then by construction, $a''$ does not lie on $L_3$. On the other hand, since $L_1,L_2$ cannot intersect at more than one point (not even in any extension), any amalgam of $A'$ and $A''$ over $A$ would need to identify $a'$ and $a''$, which is clearly impossible.

	Now, suppose that $A'$ is an arbitrary nontrivial planarisation. Then by definition we can find a sequence $A_0'=A,A_1',\ldots,A_n'=A'$ such that $A_{i+1}'\supseteq A_i'$ is an elementary planarisation for each $i<n$. Since $A'$ is nontrivial, there is a minimal $i$ such that $A_i'\subseteq A_{i+1}'$ is nontrivial. As in the preceding paragraph, we can find a trivial planarisation $A_i'\to A''$ which cannot be amalgamated with $A_{i+1}'$, and hence with $A'$, over $A_i'$. By Remark~\ref{remark:lifting_amalgams}, we conclude that $A'$ and $A''$ do not amalgamate over $A$, either. Since $A'\subseteq A''$ is trivial by construction, we are done.
\end{proof}

\begin{lm}
	\label{lemma:non_plane_has_incompatible_planarisations}
	If $A$ is a finite linear structure which is not a projective plane, then it admits two finite planarisations which cannot be amalgamated over $A$.
\end{lm}
\begin{proof}
	By Lemma~\ref{lemma:nontrivial_incomaptible}, it is enough to show that $A$ admits a nontrivial planarisation. Since planarisations are closed under composition, we may always replace $A$ by a suitable finite planarisation as needed.

	\begin{figure}
		\centering

		\begin{minipage}{0.45\textwidth}
			\centering
			\begin{tikzpicture}
				\filldraw[black] (0, 0) circle (2pt) node[anchor=east]{$a_1$};
				\filldraw[black] (2, 0) circle (2pt) node[anchor=south west]{$a_2$};
				\filldraw[black] (0, -1) circle (2pt) node[anchor=east]{$a_1'$};
				\filldraw[black] (2, -1) circle (2pt) node[anchor=north west]{$a_2'$};
				\filldraw[black] (1, 1) circle (2pt) node[anchor=east]{$b_1$};
				\filldraw[black] (1, -2) circle (2pt) node[anchor=south]{$b_2$};
				\draw plot [smooth] coordinates { (1, 1) (0, 0) (0, -1)};
				\draw plot [smooth] coordinates { (1, 1) (2, 0) (2, -1)};
				\draw plot [smooth, tension=1] coordinates { (0, 0) (2, -1) (1, -2)};
				\draw plot [smooth, tension=1] coordinates { (2, 0) (0, -1) (1, -2)};

				\draw[black] (3, -0.5) circle (2pt) node[anchor=west]{$b_3$};
				\draw plot [smooth] coordinates {(0, 0) (2, 0) (3, -0.5) };
				\draw plot [smooth] coordinates {(0, -1) (2, -1) (3, -0.5) };
				\draw[dashed] plot [smooth, tension=1] coordinates {(1, 1) (3, -0.5) (1, -2) };

			\end{tikzpicture}
		\end{minipage}\hfill
		\begin{minipage}{0.45\textwidth}
			\centering

			\begin{tikzpicture}
				\filldraw[black] (-1, 0) circle (2pt) node[anchor=east]{$a_1$};
				\filldraw[black] (0,0) circle (2pt) node[anchor=north east]{$a_2$};
				\filldraw[black] (1,0) circle (2pt) node[anchor=south west]{$a_3$};
				\filldraw[black] (0, -1) circle (2pt) node[anchor=north east]{$a_1'$};
				\filldraw[black] (1, -1) circle (2pt) node[anchor=south east]{$a_2'$};
				\filldraw[black] (0.5, 1) circle (2pt) node[anchor=west]{$b_2$};
				\filldraw[black] (1, -2) circle (2pt) node[anchor=east]{$b_1$};
				\draw plot[smooth] coordinates {(-2, 1) (-1, 0) (0, -1) (1, -2)};
				\draw plot[smooth] coordinates{(1, -2) (1, -1) (1, 0) (0.5, 1)};
				\draw plot[smooth] coordinates{(0, -1) (0, 0) (0.5, 1)};
				\draw[black] (2, 0) circle (2pt) node[anchor=west]{$b_3$};
				\draw plot[smooth] coordinates{(-1, 0) (0, 0) (1, 0) (2, 0)};
				\draw plot[smooth] coordinates{(0, -1) (1, -1) (2, 0)};
				\draw[black] (-2, 1) circle (2pt) node[anchor=east]{$b_4$};
				\draw plot[smooth] coordinates{(-2, 1) (0.5, 1) (2, 0)};
				\draw[dashed] (-2, 1)--(0, 0);
				\draw[dashed] (2, 0)--(1, -2);

			\end{tikzpicture}
		\end{minipage}

		\caption{Diagrams for the proof of Lemma~\ref{lemma:non_plane_has_incompatible_planarisations}. The unfilled dots represent points not present in the original space. On the le    ft, the dashed line represents a new nontrivial line in a planarisation of $A$ (for Case 1). On the right, the dashed lines are the new trivial lines which are parallel (for Case 2).}
		\label{figure:cases}
	\end{figure}
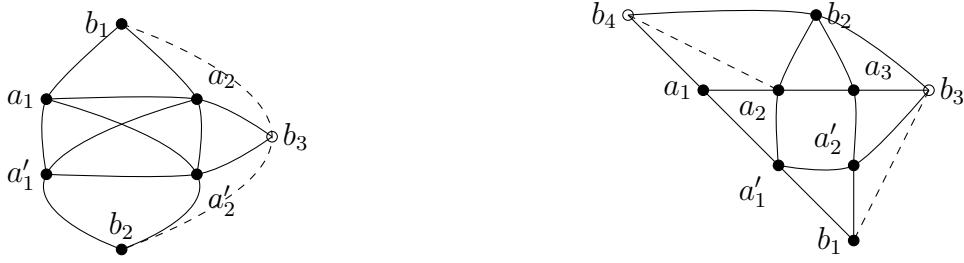
	\textbf{Case $1$:} There is a pair of trivial parallel lines in $A$ (see the diagram on the left in Figure~\ref{figure:cases}).
	Let $L=\pline(a_1, a_2)$, $L'=\pline(a_1', a_2')$ be a pair of trivial parallel lines. We may assume without loss of generality that the pairs $\pline(a_1, a_1')$, $\pline(a_2, a_2')$ and $\pline(a_1, a_2')$, $\pline(a_1', a_2)$ are non-parallel (by applying Proposition~\ref{proposition:trivial_step} if necessary) and intersect at $b_1, b_2$ respectively. Then $\pline(b_1, b_2)$ cannot intersect $L$ nor $L'$: if it did, it would have to contain one of $a_1, a_2, a_1', a_2'$, in which case it would intersect either  $\pline(a_1, a_1')$ or $\pline(a_2, a_2')$ in at least two points, a contradiction.

	Then by Proposition~\ref{proposition:intersect_three_lines}, we can find a nontrivial planarisation in which $L,L'$ and $\pline(b_1,b_2)$ are concurrent, so we are done.

	\textbf{Case $2$:} There is at least one pair of parallel lines, at least one of which is nontrivial (see the diagram on the right in Figure~\ref{figure:cases}). Let $L,L'$ be a parallel pair such that $L$ is nontrivial. Now, let $a_1, a_2, a_3$ be distinct points on $L$, while $a_1', a_2'$ are distinct points on $L'$. We may assume without loss of generality that $\pline(a_1, a_1')$ and $\pline(a_3, a_2')$ intersect at $b_1$, while $\pline(a_2, a_1')$ and $\pline(a_3, a_2')$ intersect at $b_2$.

	Extend $A$ via Proposition~\ref{proposition:trivial_step} to intersect $L,L'$ at some $b_3$. Then the lines $\pline(b_3, b_2)$ and $\pline(a_1, a_1')$ are parallel (because $\pline(b_3, b_2)$ is trivial by construction and neither $b_2$ nor $b_3$ can lie on $\pline(a_1, a_1')$, as then the latter would intersect either $\pline(a_2, a_1')$ or $L$ at two points).

	Thus, we can apply Proposition~\ref{proposition:trivial_step} again to obtain a new point $b_4$ at which $\pline(b_3, b_2)$ and $\pline(a_1, a_1')$ intersect. Then the lines $\pline(b_4, a_2)$ and $\pline(b_1, b_3)$ are parallel and trivial, so we reduce to Case 1.
\end{proof}

The following is the main technical lemma of this section.
\begin{lm}
	\label{lemma:aplanar_amalgamation}
	If $f_1\colon A\to B_1$ and $f_2\colon A\to B_2$ are embeddings of linear spaces, at least one of which is aplanar, then $f_1$ and $f_2$ can be amalgamated.
\end{lm}
\begin{proof}
	Consider $C=B_1\sqcup_AB_2$ (the union of $B_1, B_2$ so that $A = B_1 \cap B_2$) with the structure defined as follows:
	\begin{itemize}
		\item
		$B_1$ and $B_2$ are naturally embedded,
		\item
		$c_1,c_2,c_3\in C$, not all in the same $B_i$, are collinear when they are distinct and there is a line $L$ in $A$ such that each $b_i$ is in the (unique) extension of $L$ to a line in the corresponding $B_i$.
	\end{itemize}
	As soon as we check that this describes a linear space structure on $C$, it is clear that $C$ with this structure is an amalgam of $B_1$ and $B_2$ over $A$.

	Since $B_1, B_2$ each is a linear space, it suffices to check we have transitivity for the collinearity relation (symmetry and irreflexivity are obvious). Suppose $c_1,c_2,c_3,c_4\in C$ are such that $c_1,c_2,c_3$ and $c_1,c_2,c_4$ are collinear. We need to check that $c_1,c_3,c_4$ are collinear.
	\begin{enumerate}[label={\textbf{Case} \arabic*:}, wide]
		\item
		$c_1, c_2\in B_1$ or $c_1,c_2\in B_2$. We may assume without loss of generality that the former holds. If $\pline^{B_1}(c_1,c_2)\cap A$ has fewer than two points, then it follows from the definition of collinearity in $C$ that $c_3,c_4\in B_1$, so the conclusion is clear. Otherwise, $L=\pline^{B_1}(c_1,c_2)\cap A$ is a line in $A$, and it is not hard to see $c_3$ and $c_4$ necessarily lie on the extension of $L$ to the corresponding $B_i$, which yields the conclusion: e.g.\ if $c_3\in B_2$, then by definition of $\coll(c_1,c_2,c_3)$, there is a line $L'$ in $A$ on whose extensions $c_1,c_2,c_3$ all lie, but then $L'$ has to equal $\pline^{B_1}(c_1,c_2)\cap A=L$.
		\item
		Otherwise, we may assume without loss of generality that $c_1\in B_1\setminus A$, $c_2\in B_2\setminus A$.

		\begin{clm*}
			In this case, there is a unique line $L$ in $A$ such that $c_1$ and $c_2$ both lie on (the extensions in the appropriate spaces of) $L$.
		\end{clm*}

		\begin{clmproof}
			By the definition of collinearity, there is at least one such line. If there were two such lines, then both $c_1$ and $c_2$ would be the intersections of those two lines in $B_1$ and $B_2$, respectively, witnessing that neither $B_1$ nor $B_2$ is an aplanar extension of $A$, counter to the hypothesis of the lemma.
		\end{clmproof}

		Now, let $L$ be the line we have by the claim. Similarly to Case 1, both $c_3$ and $c_4$ necessarily have to lie on (the appropriate corresponding extensions of) $L$: for example, if $c_3\in B_1$, then (because $c_2\notin B_1$), $\coll(c_1, c_2, c_3)$ implies that there is a line $L'$ in $A$ on whose extensions $c_1, c_2$ and $c_3$ all lie, and by Claim, $L'=L$, and likewise when $c_3\in B_2$ and for $c_4$. The conclusion easily follows.\qedhere
	\end{enumerate}
\end{proof}

\begin{prop}
	\label{proposition:projective_iff_amalgamation_base}
	A finite linear space is an amalgamation base (in the class of finite linear spaces) if and only if it is a projective plane.
\end{prop}
\begin{proof}
	Let $A$ be a linear space. If it is not a projective plane, then by Lemma~\ref{lemma:non_plane_has_incompatible_planarisations}, it is not an amalgamation base.

	Conversely, if $A$ is a projective plane, then any embedding of $A$ is aplanar, so by Lemma~\ref{lemma:aplanar_amalgamation} we conclude that it is an amalgamation base.
\end{proof}

\begin{prop}
	\label{proposition:amalgamable_implies_complete_planarisation}
	If $A\to B$ is an amalgamable embedding of linear spaces, then the planar closure of $A$ in $B$ is a projective plane. In particular, $A$ admits a finite amalgamable extension if and only if $A$ admits a finite planarisation.
\end{prop}
\begin{proof}
	Suppose that $A\to B$ is an amalgamable embedding. We may assume for simplicity that $A\subseteq B$. Let $\bar A$ be the planar closure of $A$ in $B$. We will show that $\bar A$ is an amalgamation base, whence by Proposition~\ref{proposition:projective_iff_amalgamation_base} it will follow that it is a projective plane.

	Let $A_1, A_2$ be arbitrary extensions of $\bar A$. Since $\bar A$ is the planar closure of $A$ in $B$, it follows that $\bar A\subseteq B$ is aplanar, so by Lemma~\ref{lemma:aplanar_amalgamation}, we can amalgamate $B$ and $A_i$ over $\bar A$ for $i=1,2$ to obtain $B_1$ and $B_2$. Since $A\subseteq B$ is amalgamable, $B_1$ and $B_2$ have an amalgam $C$ over $A$, and by Remark~\ref{remark:lifting_amalgams}, $C$ is also an amalgam of $B_1$ and $B_2$ over $\bar A$, and hence also an amalgam of $A_1$ and $A_2$ over $\bar A$. Since $A_1,A_2$ were arbitrary, we are done.

	
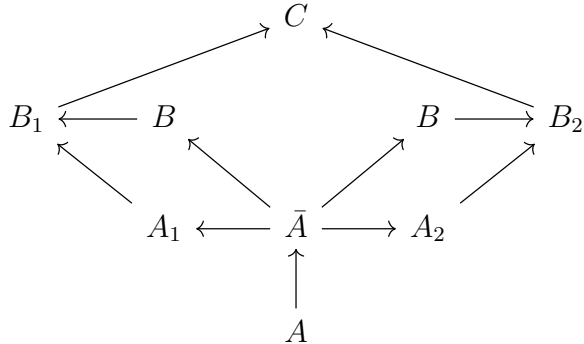
\begin{center}
		\begin{tikzcd}
			&&C&\\
			B_1
			\ar[rru]&
			B
			\ar[l]&&B\ar[r]&
			B_2
			\ar[llu]
			\\
			&A_1
			\ar[ul]
			&\bar A
			\ar[ul]\ar[ur]\ar[r]\ar[l]
			&A_2
			\ar[ru]
			\\
			\,&&A
			\ar[u]&
		\end{tikzcd}
	\end{center}
\end{proof}

\begin{tw}
	\label{theorem:wap_cap_planes}
	Let $A$ be a finite linear space. Then the following are equivalent:
	\begin{itemize}[itemsep=0pt]
		\item
		$A$ embeds into a finite projective plane,
		\item
		$A$ embeds into an amalgamation base (in the class of finite linear spaces),
		\item
		$A$ admits an amalgamable embedding (in the same class).
	\end{itemize}
	Consequently, the following are equivalent:
	\begin{itemize}[itemsep=0pt]
		\item
		Conjecture $\dagger$, i.e. every finite linear space can be embedded into a finite projective plane,
		\item
		the class of finite linear spaces has the cofinal amalgamation property,
		\item
		the class of finite linear spaces has the weak amalgamation property.
	\end{itemize}
\end{tw}
\begin{proof}
	The first two bullets are clearly equivalent by Proposition~\ref{proposition:projective_iff_amalgamation_base}. The second and third bullet are equivalent by Proposition~\ref{proposition:amalgamable_implies_complete_planarisation}.

	The ``consequently'' part easily follows, since CAP just means that every finite linear space can be embedded into an amalgamation base, while WAP means that every such space admits an amalgamable embedding.
\end{proof}

\begin{uwgi}
    Note that by Lemma~\ref{lemma:non_plane_has_incompatible_planarisations} (or alternatively, by Lemma~\ref{lm:homogenous_plane_degree_4}), it follows that the class of finite linear spaces does not have the full amalgamation property.
\end{uwgi}

\begin{uwgi}
	Most of the results in this section remain true (after some obvious adjustments) if we drop the hypotheses that the linear spaces be finite, with essentially the same proofs (possibly requiring some transfinite inductive arguments).
\end{uwgi}

On the other hand, countable linear spaces embed into countable projective planes, therefore they have the cofinal amalgamation property. This is explored in the next section.

We finish this section with the following corollary to the results above together with a non-universality result from~\cite{KraKub}.

\begin{tw}\label{THMJuniwa}
	Assume there exists a family $\Yu$ consisting of strictly less than continuum many countable linear spaces, such that every countable linear space embeds into some $U \in \Yu$. Then Conjecture {\rm(\pfp)} is true.
\end{tw}

\begin{pf}
	By~\cite[Corollary 6.3]{KraKub}, the class of finite linear spaces has the weak amalgamation property. By Theorem~\ref{theorem:wap_cap_planes}, every finite linear space embeds into a finite amalgamation base, which is a projective plane (as long as it contains four independent points).
\end{pf}

\begin{wn}
	If there exists a universal countable linear space then Conjecture {\rm(\pfp)} is true.
\end{wn}

Let us note that the validity of Conjecture (\pfp) does not imply any universality results, see~\cite[Example 6.5]{KraKub}.

\section{Universal projective planes}

For the moment it is not clear what the universality number of countable linear spaces is. We prove that if the continuum hypothesis holds, there exists a universal projective plane of cardinality $\aleph_1$.
It is actually the \fra\ limit~\cite{Hodges, Kub40} of the class of all countable projective planes.
The following simple lemma is well known.

\begin{lm}\label{LMghyio}
	Let $\bX = \pair{X}{\El}$ be a linear space.
	Then there exists a projective plane of cardinality $ \loe |X|+\aleph_0$ containing an isomorphic copy of $\bX$.
\end{lm}

\begin{pf}
	First, extend $\bX$ so that it becomes non-degenerate (it is enough to add one or two extra points and extend the lines properly). Next, extend it to $\bX_1$ so that every two lines from $\bX$ intersect in $\bX_1$. This can be done by adding an extra point for each pair of parallel lines. In particular, the size of $\bX_1$ is not greater than the maximum of the size of $\bX$ and $\aleph_0$.
	Repeat this procedure infinitely many times, obtaining a plane $\bX_\infty$ which is closed, therefore projective.
\end{pf}

\begin{lm}\label{LMgiiuet}
	Assume $\bX = \pair X \El$ is an uncountable projective plane.
	Given a countable $Y \subs X$, there exists a countable projective subplane $\tilde{Y}$ of $\bX$ containing $Y$.
\end{lm}

\begin{pf}
	Apply the same procedure as in the previous lemma, working inside $\bX$
    (equivalently, we can just take the planar closure of $Y$ in $\bX$).
\end{pf}

\begin{tw}\label{THMFajwPointFri}
	Assume the continuum hypothesis. There exists a unique projective plane $\bV$ of cardinality $\aleph_1$ with the following extension property:
	\begin{enumerate}
		\item[{\rm(E)}] Given a countable closed plane $B$ and its closed subplane $A$, every embedding of $A$ into $\bV$ extends to an embedding of $B$ into $\bV$.
	\end{enumerate}
	Furthermore:
	\begin{enumerate}[itemsep=0pt]
		\item[{\rm(H)}] Every isomorphism between countable closed subplanes of $\bV$ extends to an automorphism of $\bV$.
		\item[{\rm(U)}] Every plane of cardinality $\loe \aleph_1$ embeds into $\bV$.
	\end{enumerate}
\end{tw}

Recall that a plane is closed if every two of its lines intersect.

\begin{pf}
	The first part (the existence and uniqueness) is an instance of the general theory of \fra\ limits (see e.g.~\cite{Kub40}), thanks to the amalgamation property
    (which follows from Lemmas~\ref{LMgiiuet} and \ref{lemma:aplanar_amalgamation}) and the cardinal arithmetic assumption implying that there are at most $\aleph_1$ many
    isomorphisms
    types and embeddings between countable closed planes. Property (H), homogeneity, again follows from the general theory.

	Concerning Property (U), universality, the general theory gives it for all closed planes of cardinality $\loe \aleph_1$.
	On the other hand, thanks to Lemma~\ref{LMghyio}, every plane of cardinality $\loe \aleph_1$ can be extended to a projective plane of cardinality $\aleph_1$ and, thanks to Lemma~\ref{LMgiiuet} and easy transfinite induction, a projective plane of cardinality $\aleph_1$ can be decomposed into a continuous $\omega_1$-chain of countable projective subplanes. Thus, the general theory gives the full universality, that is, Property (U).
\end{pf}

\subsection{Final remarks}

It turns out that if finite linear spaces fail the weak amalgamation property, which according to our results is equivalent to the failure of Conjecture (\pfp), then the smallest cardinality of a universal family of countable projective planes is the continuum, see Theorem~\ref{THMJuniwa} or~\cite[Corollary 6.3]{KraKub}.

On the other hand, if Conjecture (\pfp) is true, the general theory of generic limits yields a countable projective plane $\bW$ that is generic in the sense of the natural game, where two players alternately build bigger and bigger finite linear spaces, simply by adding new points and new lines. Either of the two players has a strategy leading to a linear space isomorphic to $\bW$, obtained as the union of the chain built by the players.
This game, called an \emph{abstract Banach-Mazur game}, makes sense in arbitrary mathematical structures and even in arbitrary categories, see~\cite{KraKub} and \cite{Kub61} for details.
The generic plane $\bW$ (assuming it exists) would be locally finite, in the sense that every finite subset is contained in a finite projective subplane. Hence, it would not contain the projective plane $\proplane{\Qyu}$, therefore it would not be universal in the class of all countable projective planes.

{From} the previous section we know that, assuming the continuum hypothesis, there exists a universal projective plane $\bV$ of the smallest uncountable cardinality, due to Theorem~\ref{THMFajwPointFri}. Furthermore, homogeneity with respect to countable closed subplanes makes $\bV$ unique up to isomorphism. Indeed, if $A \subs B$ are countable closed planes and $\map e A \bV$ is an embedding then, by universality, there exists an embedding $\map f B \bV$ and, by homogeneity, there is an automorphism $\map h \bV \bV$ such that $f\rest A = h \cmp e$; finally $h^{-1} \cmp f$ extends $e$. This shows that $\bV$ has property (E) of Theorem~\ref{THMFajwPointFri}.
Note that $\bV$ cannot be algebraic, as it contains all countable projective planes including the ``pathological'' (non-Desarguesian) ones. Nevertheless, $\bV$ contains copies of all projective planes of the form $\proplane{K}$, where $K$ is any field (or, more generally, a division ring) of cardinality at most the continuum. That includes the field of complex numbers and all its subfields, as well as all finite and countable fields of positive characteristic.

Summarizing, the following questions seem to be relevant here.

\begin{pyt}
	Does there exist a universal projective plane of cardinality continuum, without any extra set-theoretic assumptions?
\end{pyt}

\begin{pyt}
	What is the universality number of the class of all countable projective planes?
\end{pyt}

By Theorem~\ref{THMJuniwa}, the answer to the second question is easy (the continuum) if there exists a finite linear space not embeddable into any finite projective plane (in case Conjecture (\pfp) is false). Otherwise, we only have a generic countable plane $\bW$, mentioned above, which is not universal. It is stil possible that there exists another countable projective plane containing copies of all countable linear spaces.


\begin{thebibliography}{99}

\bibitem{BaBe}
{\sc L. M. Batten, A. Beutelspacher}, {\it The theory of finite linear spaces}.
Cambridge University Press, Cambridge, 1993

\bibitem{Beutel}
{\sc A. Beutelspacher}, {\it Projective planes}. In: {\it Handbook of Incidence Geometry}, 107–136, North-Holland, Amsterdam, 1995

\bibitem{Buekenhout}
{\sc F. Buekenhout}, {\it Questions about linear spaces}, Discrete Math. {\bf129} (1994) 19--27

\bibitem{CoxPP}
{\sc H.S.M. Coxeter}, {\it Projective geometry}. Revised reprint of the second (1974) edition. Springer-Verlag, New York, 1994

\bibitem{Fraisse} {\sc R. \fra}, {\it Sur l'extension aux relations de quelques propri\'et\'es des ordres}, Ann. Sci. Ecole Norm. Sup. (3) {\bf71} (1954) 363--388

\bibitem{Hodges}
{\sc W. Hodges}, {\it Model theory}. Encyclopedia of Mathematics and its Applications, 42. Cambridge University Press, Cambridge, 1993

\bibitem{Ivanov}
{\sc A. Ivanov}, {\it Generic expansions of $\omega$-categorical structures and semantics of generalized quantifiers}, J. Symbolic Logic 64 (1999) 775--789


\bibitem{KraKub} {\sc A. Krawczyk, W. Kubi\'s}, {\it Games on finitely generated structures}, Ann. Pure Appl. Logic 172 (2021), Paper No. 103016, 13 pp.

\bibitem{Kub61}
{\sc W. Kubi\'s}, {\it Weak \fra\ categories}, Theory and Applications of Categories Vol. 38, 2022, No. 2, 27--63

\bibitem{Kub40}
{\sc W. Kubi\'s}, {\it \fra\ sequences: category-theoretic approach to universal homogeneous structures}, Ann. Pure Appl. Logic 165 (2014) 1755--1811


\bibitem{Shult}
{\sc E. E. Shult}, {\it Points and lines. Characterizing the classical geometries}. Universitext. Springer, Heidelberg, 2011

\end{thebibliography}
\end{document}